\documentclass[10pt,twoside]{amsart}
\usepackage{float}
\usepackage{amssymb, amsmath, mathrsfs, mathtools, amsfonts, amsthm, amscd, yfonts}
\usepackage{xcolor}
\usepackage{microtype}
\usepackage{hyperref} 
\usepackage{multicol}
\usepackage{tikz-cd}
\usepackage{hyphenat}
\usepackage{lineno}
\theoremstyle{plain}
\newtheorem{Theorem}{Theorem}
\newtheorem{Proposition}[Theorem]{Proposition}
\newtheorem{Lemma}[Theorem]{Lemma}
\newtheorem{Corollary}[Theorem]{Corollary}

\theoremstyle{definition}
\newtheorem{Definition}[Theorem]{Definition}
\theoremstyle{remark}
\newtheorem{Remark}[Theorem]{Remark}
\newtheorem{Example}[Theorem]{Example}

\newcommand{\cM}{\mathcal{M}}

\newcommand{\rR}{\mathscr{R}}

\title{QUANTUM REGULARITY of FINITE DIMENSIONAL
SEMISIMPLE ALGEBRAS}

\author[Y. Bahturin]{Yuri Bahturin$^{1,2}$}
\address{1. Department of Mathematics and Statistics, Memorial University of Newfoundland, St. John's, NL, Canada, A1C 5S7}
\email{bahturin@mun.ca}
\author[L. Centrone]{Lucio Centrone$^2$}
\address{2. Dipartimento di Matematica, Universit\`a degli Studi di Bari Aldo Moro, Via Edoardo Orabona, 4, 70125 Bari, Italy}
\email{lucio.centrone@uniba.it}
\author[K. Pereira]{Kau\^e Pereira$^3$}
\address{3. IMECC, UNICAMP, Rua S\'ergio Buarque de Holanda 651, 13083-859 Campinas, SP, Brazil}
\email{k200608@dac.unicamp.br}
\thanks{K. Pereira was supported by FAPESP Grant 2023/01673-0, and FAPESP Grant 2025/03763-2}



\begin{document}
	
	\begin{abstract}
	In this paper, we establish a necessary and sufficient criterion for a finite-dimensional semisimple algebra over an algebraically closed field of characteristic $0$ to admit a regular quantum commutative decomposition. We apply this characterization to group algebras and show that a finite group algebra admits such a decomposition if and only if the underlying group is a peak group, that is, a finite group whose irreducible character degrees have a greatest element with respect to divisibility. We investigate the structure of peak groups and establish several criteria for their solvability in terms of the prime divisors of their largest irreducible character degree. In particular, we show that several important classes of groups are peak groups, while supersolvability alone does not imply the peak property. Finally, we show how to construct a peak group from any finite group. 
	\end{abstract}
	\maketitle
	
	\section{Introduction}
	
	Regular gradings on associative algebras were introduced by Bahturin and
	Zaicev in \cite{BahturinZaicev} as a natural framework for describing algebras whose homogeneous
	components satisfy a commutation rule defined by a
	skew-symmetric bicharacter. Besides their intrinsic algebraic interest,
	regular gradings play an important role in the theory of polynomial
	identities, graded identities, Hopf actions, and the structure theory of
	finite-dimensional algebras. We refer the reader to
	\cite{BahturinZaicev, Kochetov.Elduque.Book} for general background.
	
	Regular gradings are a particular case of  regular quantum commutative decompositions introduced by Regev and Seeman in \cite{regevz2}. Their approach replaces  group gradings by  arbitrary vector space decompositions while
	retaining the fundamental quantum commutation relations between their
	components. Of course, every regular grading is a regular quantum commutative
	decomposition, but the latter class is considerably larger and provides
	a  framework for the study of various families of associative
	algebras.
	
	One of the main problems in this context is to determine which
	finite-dimensional algebras admit such decompositions. In
	\cite{B.L.K}, several necessary conditions have been obtained. In particular,
	let
	\[
	R=B_1\oplus\cdots\oplus B_t\oplus J(R)
	\]
	be the Wedderburn--Malcev decomposition of a regular quantum commutative finite-dimensional algebra $R$ of quantum length $m$,
	where $B_i$ are the simple components and $J(R)$ is the Jacobson radical. We have shown that the dimensions of these components depend strongly on the quantum length $m$. In the semisimple case, it was proved that,
	if the regular quantum commutative decomposition is minimal, then the maximum of the dimensions $\dim B_i$ has a nontrivial common divisor with any $\dim B_i$, $i=1,\ldots t$.
	
	In this paper, we give a complete answer to the main question in the case of semisimple algebras over  algebraically closed fields.
	
	Our main result is the following.
	
	\begin{Theorem}
		\label{main.theorem}
		Let $R=M_{n_{1}}(K)\oplus \cdots \oplus M_{n_{r}}(K)$ be a finite-dimensional semisimple algebra over an algebraically closed field $K$. Then, $R$ admits a regular quantum commutative decomposition if and only if there exists $1\leq s\leq r$ such that $n_{i}\mid n_{s}$, for all $1\leq i\leq r$. 
	\end{Theorem}
	
	The proof of Theorem 1 combines the theory of non-degenerate skew-symmetric
	bicharacters on finite abelian groups with structural properties of
	regular gradings on matrix algebras. The key step consists in showing
	that whenever the decomposition matrix associated with one regular
	grading occurs as a submatrix of another, {the corresponding grading
	groups satisfy a divisibility relation on their orders.} This immediately translates into a representation-theoretic characterization of group algebras of finite groups. Indeed, by Maschke's Theorem, a group algebra $KG$ of a finite group $G$ over an algebraically closed field $K$ whose characteristic is zero is semisimple and it is the direct sum of matrix algebras whose orders are equal to the degrees of irreducible characters of $G$. Consequently, the main result of the paper shows that the existence of a regular quantum commutative decomposition
	of $KG$ is completely determined by the divisibility relations among the
	degrees of irreducible characters  of $G$.
	
	This motivates the introduction of the class of \emph{peak groups}, namely finite groups whose irreducible character degrees possess the greatest element with respect to the divisibility order. We prove that a finite group algebra admits a regular quantum commutative decomposition if and only if the underlying group is a peak group. We further investigate this class by showing, for instance, that all finite nilpotent groups and all metacyclic groups are peak groups and therefore give rise to regular quantum commutative decompositions of their group algebras. These two families belong to the class of supersolvable groups. Nevertheless, we show that supersolvability alone does not imply the peak property. We also establish several solvability criteria for peak groups whose maximal irreducible character degree is divisible by at most three distinct primes. Furthermore, by working within the class of Frobenius groups, we construct a non-solvable peak group. Finally, in the last subsection of the paper, we show how to construct a peak group from any finite group. This result, together with Zassenhaus's theorem, allows us to exhibit a peak group that is neither solvable nor a Frobenius group.
	
	\section{Background}

	Throughout this paper, $K$ will denote an algebraically closed field of characteristic $0$. By an algebra, we mean an associative unital $K$-algebra.

    For any $n>1$ we will write $\textbf{U}_n$ to denote the subgroup of $n$-th roots of unity in the multiplicative group $K^\ast$.

	\begin{Definition}\cite[Definition 2.3]{regevz2}\label{dQCD} Let $R$ be an associative algebra over a field $K$. Consider a vector space decomposition of $R$ into nonzero subspaces 
		\begin{equation}\label{eQCD}
			\mathscr{R}\colon\quad R=\bigoplus_{s\in S} R_s.
		\end{equation}
		We call $\mathscr{R}$ a \textit{regular quantum commutative decomposition} of $R$ with support $S$  if the following hold:
		\begin{enumerate}
			\item[(a)] for any natural $n$ and any, possibly equal, $s_{1},\ldots, s_{n}\in S$, \[R_{s_1}\cdots R_{s_n}\ne\{ 0\}.\]
			\item[(b)] there is a function $\beta:S\times S\to K^\ast$ such that for each $s,t\in S$, $a\in R_s$ and $b\in R_t$,  one has 
			\[
			ab=\beta(s,t)ba.
			\]  
		\end{enumerate}
		The function $\beta$ is called the \textit{regular quantum commutative function of $R$ with respect to the decomposition $\mathscr{R}$}. If $S$ is finite, then $m=|S|$ is called   the \textit{quantum length} of $\mathscr{R}$. An algebra $R$ admitting $\rR$ and satisfying (b) only is called \emph{$(\rR,\beta)$-commutative}. If the \emph{datum} $(\rR,\beta)$ is fixed, we simply say that $R$ is a \emph{regular quantum commutative algebra} and we write $R=(\rR,\beta)$.
	\end{Definition}
	The matrix $\cM^{R}=(\beta(i,j))$ is called the \textit{quantum decomposition matrix} (or just the \textit{decomposition matrix}) of a regular quantum commutative algebra $R=(\mathscr{R},\beta)$. 
	
	\begin{Remark}\label{graded} Let $S$ be an abelian group $G$  and  let $\rR$ be a grading by $G$. Then $\beta$ is a skew-symmetric bicharacter of $G$, and $R$ is a $G$-graded  $\beta$-commutative algebra in the usual sense. 
		In this case, we say that $\rR$ is a \textit{$G$-regular grading} with a skew-symmetric bicharacter $\beta$. We also say that $R$ is a $\beta$-commutative $G$-regular graded algebra and we write $\cM_{\beta}:=\cM^{R}$.
	\end{Remark}
	\begin{Definition} Let $R=(\rR,\theta)$ be a regular quantum commutative algebra. We say that the regular quantum commutative decomposition of $R$ is \textit{minimal} if the columns (equivalently, rows) of the quantum decomposition matrix $\cM^{R}$ are pairwise different.
	\end{Definition}

    \begin{Example}
        In the case of Remark \ref{graded}, $\rR$ is minimal if and only if $\beta$ is non-degenerate in the usual sense, i.e., if $\beta(g,x)=1$ for all $x\in G$, then $g=e$, or equivalently, $\det \cM_\beta\neq 0$ (see \cite{Eli1} and \cite{B.L.K} ).
    \end{Example}
	
	\begin{Definition}\label{submatrix} Let $G$ and $H$ be finite abelian groups, and let $\beta$ and $\gamma$ be skew-symmetric bicharacters of $G$ and $H$, respectively. We say that $\cM_\beta$ is a \emph{submatrix} of $\cM_\gamma$ if there exists a function $f\colon H\rightarrow G$ such that $\gamma$ is a restriction of $\beta$ to $f(H)$, in other words,
		\[
		\gamma(h_{1},h_{2})=\beta(f(h_{1}),f(h_{2})),\quad \text{for all}\quad h_{1},h_{2}\in H.
		\]
	\end{Definition}
	
	\begin{Remark}\label{Remark1} Given a finite abelian group $G$ and $\beta\colon G\times G\rightarrow K^{\ast}$ a skew-symmetric bicharacter of $G$, it is well known that $\beta(g,h)\in \textbf{U}_n$, $n=|G|$, for all $g$, $h\in G$. It follows immediately that any two subgroups $H$ and $K$ of $G$ with coprime orders are \textit{orthogonal} with respect to $\beta$: if $x\in H$ and $y\in K$, then $\beta(x,y)=1$. 
	\end{Remark}

	\begin{Theorem}[Theorem 27 and Theorem 29 of \cite{B.L.K}]
		Let $R=(\rR,\beta)$ be a regular quantum commutative algebra with support $S$.
		\begin{enumerate}
			\item[(a)] If $K$ is an arbitrary field and $\dim R<\infty$, then $\beta(s,t)\in \textbf{U}_n$, for all $s,t\in S$, where $n=\dim R$.
			\item[(b)] If $\operatorname{char}(K)=0$ and $K$ is an algebraically closed field, then all $\beta(s,t)\in \textbf{U}_k$ for some $k\in \mathbb{N}$.
		\end{enumerate}
	\end{Theorem}

	In \cite{B.L.K}, we can find the following result, with a slightly different notation.
	\begin{Proposition}\cite[Proposition 46]{B.L.K}\label{prop from 6}
		\label{necessary.condition}
		Let $R=(\rR,\theta)$ be a finite-dimensional regular quantum commutative algebra. Consider its Wedderburn--Malcev decomposition, $R = B +J(R)$, where $B = B_{1} \oplus \cdots \oplus B_{k}$ is the direct sum (as algebras) of simple subalgebras of $R$. If $m$ is the quantum length of $\mathscr{R}$, and $d_{i}:=\dim B_{i}$, $1\leq i\leq k $, then
		\begin{enumerate}
			\item[(i)] For all $1\le i\le k$, $d_i\leq m$.
    \item[(ii)] If $d_\ell$ is the greatest among all $d_i$, then $\gcd(d_{i},d_{\ell})>1$, $1\leq i\leq k $. 
			\item[(iii)] If $\mathscr{R}$ is minimal and $R$ semisimple, say $R=B_{1} \oplus \cdots \oplus B_{k}$, then \[\max\{\dim B_1,\ldots,\dim B_k\}=m.\]
		\end{enumerate}
	\end{Proposition}


	\section{Criterion of quantum regularity}

	Throughout this section we will write $C_n$ to denote the multiplicative cyclic group of order $n$. 
	
	\subsection{Sub-decomposition matrices}
A complete classification of skew-symmetric bicharacters on finite abelian groups was given in \cite{zolotykh1997commutation}. We briefly list here some facts which we will use in what follows. 

    Given skew-symmetric bicharacters $\beta_1,\ldots,\beta_k$ on the abelian  groups $G_1,\ldots, G_k$, one can define a skew-symmetric character $\beta$ on their direct product $G=G_1\times\cdots\times G_k$ by setting 
     \begin{equation}\label{productofbicharacters}
    \beta(a_1\cdots a_k,b_1\cdots b_k)=\beta_1(a_1,b_1)\cdots\beta_k(a_k,b_k)\mbox { where } a_i,b_i\in G_{i}\mbox{ for }1\le i\le k.    \end{equation}
   We denote such \textit{external} product by $\beta=\beta_1\otimes\cdots\otimes \beta_k$. It is easy to see that the matrix $\cM_\beta$ of $\beta$ is the Kronecker product  of the matrices $\cM_{\beta_i}$ for $\beta_i$, $i=1,\ldots,k$:
   \[
   \cM_\beta=\cM_{\beta_1}\otimes \cdots\otimes \cM_{\beta_k}.
   \]
   Note that $\beta$ is non-degenerate if and only if all $\beta_1,\ldots,\beta_k$ are non-degenerate.

    As a partial converse, let $G$ be a finite abelian group whose order is the product of nonzero powers of pairwise different numbers $p_1,\ldots,p_k$. Then  $G=G_1\times\cdots\times G_k$, where each $G_i$ is a Sylow $p_i$-subgroup of $G$. Given a skew-symmetric bicharacter $\beta$ on $G$, we denote its restrictions $\beta_i$ to $G_{i}$ by  $\beta_i$. Then the value of $\beta$ can be computed using (\ref{productofbicharacters}). In particular, the decomposition $G=G_1\times\cdots\times G_k$ is  \textit{orthogonal}, in the sense that $\beta(G_i, G_j)=1$, for all $1\le i\ne j\le k$. With some abuse of notation, we write $\beta=\beta_1\otimes\cdots\otimes \beta_k$ also in the case of \textit{internal} direct products of groups.

    Now suppose we have two finite abelian groups, each endowed with a non-degenerate skew-symmetric bicharacter, $\beta$ for $G$, and $\gamma$ for $H$. As in Definition \ref{submatrix}, we consider a bicharacter preserving map $f:H\to G$,  that is, for all $x,y\in H$ we have $\beta(f(x),f(y))=\gamma(x,y)$. As noted, for each prime divisor $p$ of both orders of $G$ and $H$, the restriction $\beta_p$ (resp., $\gamma_p$) of $\beta$ (resp., $\gamma$) to a Sylow $p$-subgroup $G_p$ of $G$ (resp., $H_p$ of $H$) is non-degenerate. Now if  $\pi:G\to G_p$ is a natural projection of $G$ onto $G_p$ then one can consider the map $f_p:H_p\to G_p$ which is the restriction of $\pi\circ f:H\to G_p$ to $H_p$.
	
	\begin{Lemma} 
		\label{Lemma.1}
		For each prime $p$, which is the divisor of both $G$ and $H$, the map $f_p: H_p\to G_p$ is an injective bicharacter preserving map.
	\end{Lemma}
	\begin{proof} First we note that any bicharacter preserving map $f:H\to G$ of $H$ with non-degenerate bicharacter $\gamma$ into $G$ with non-degenerate bicharacter $\beta$ is injective. Indeed, if $x_1,x_2\in H$ are such that $f(x_1)=f(x_2)$, then, for any $x\in H$, we should have
		\[
		\gamma(x,x_1)=\beta(f(x),f(x_1))=\beta(f(x),f(x_2))=\gamma(x,x_2).
		\]
		 Since $x \in H$ is arbitrary and  $\gamma$ is non-degenerate, we get $x_1 = x_2$. Thus, $f$ is injective. 

         To prove that $f_p$ is bicharacter preserving, we take $x,y\in H_p$. We need to show that $\beta_p(f_p(x), f_p(y))=\gamma_p(x,y)$. Now $f(x)=f_p(x)u$ and $f(y)=f_p(y)v$, where $u,v\in G_p'$. By definitions and since $\beta(G_p,G_p')=1$, we have
         \begin{align}
        \gamma_p(x,y)&=\gamma(x,y)=\beta(f(x),f(y))=\beta(f_p(x)u,f_p(y)v)\label{betagamma}\\&=\beta(f_p(x),f_p(y))\beta(u,v)=\beta_p(f_p(x),f_p(y))\beta(u,v). \nonumber  
         \end{align}

      As noted earlier (\ref{Remark1}), if $|G_p|=p^k$ and $|H_p|=p^r$, then $\gamma_p(x,y)\in \textbf{U}_{p^r}$, $\beta_p(f_p(x),f_p(y))\in\textbf{U}_{p^k}$ and $\beta(u,v)\in \textbf{U}_{|G_p'|}$. Since $|G_p'|$ is coprime to $p$, we must have $\beta(u,v)=1$. In this case, (\ref{betagamma}) implies
      \[
\gamma_p(x,y)=\beta_p(f_p(x),f_p(y)).
      \]
      So that $f_p$ is bicharacter preserving. By the first part of the proof, $f_p$ is injective, and so the proof is complete.
        \end{proof}
    \begin{Remark} Let $G$ and $H$ be finite abelian groups, each endowed with a non-degenerate skew-symmetric bicharacter, denoted by $\beta$ on $G$ and by $\gamma$ on $H$. Suppose that $f\colon H\rightarrow G$ is a bicharacter-preserving map. Then, the proof of Lemma~\ref{Lemma.1} shows that if $p\mid |H|$, then $p\mid |G|$. In other words, the set of prime divisors of $|H|$ is contained in the set of prime divisors of $|G|$.
    \end{Remark}

      \begin{Corollary}\label{Bicharpres}
      Let $G$ and $H$ be finite abelian groups endowed by respective non-degenerate bicharacters $\beta$ and $\gamma$. If $f:H\to G$ is a bicharacter preserving map, then $|H|$ divides $|G|$.    
      \end{Corollary}  
	
	\begin{proof} By Lemma \ref{Lemma.1}, the order of each Sylow $p$-subgroup of $H$ divides the order of the Sylow $p$-subgroup of $G$. Since the order of a finite group is the product of the orders of its Sylow subgroups, our claim follows. 
	\end{proof}

	\begin{Remark} Let $G$ and $H$ be finite abelian groups endowed with  non-degenerate skew-symmetric bicharacters $\beta$ and $\gamma$, respectively. Assume that $\gcd(|G|,|H|)=1$. If $f:H\to G$ is a bicharacter preserving map then  for any $x,y\in H$, we have
    \[
    \gamma(x,y)=\beta(f(x),f(y)).
    \]	
	Now the left-hand side of this equation is in $\textbf{U}_{|H|}$ whereas the right-hand side in 	$\textbf{U}_{|G|}$. Thus, $\gamma$ is not non-degenerate, unless $H=\{ e\}$.
	\end{Remark}
	\subsection{Realizing regular gradings}
	
	The following result is essentially a summary of some key points in \cite{Kochetov.Elduque.Book}.
	\begin{Lemma}
		\label{realization}
		Let $n\in \mathbb{N}$ and $\ell_{1}$,\dots, $\ell_{r}\in \mathbb{N}$ be such that $\ell_{1}\cdots \ell_{r}=n$. Given the finite abelian group $G:=C_{\ell_{1}}^{2}\times \cdots \times C_{\ell_{r}}^{2}$, we have the following:
		\begin{enumerate}
			\item[(a)] There exists a non-degenerate skew-symmetric bicharacter $\beta_{G}\colon G\times G\rightarrow K^{\ast}$ of $G$.
			\item[(b)] The matrix algebra $M_{n}(K)$ admits a structure of $G$-regular grading with skew-symmetric bicharacter $\beta_G$.
			\item[(c)] Given $1\leq s\leq r$, $k_{1},\ldots, k_{s}\in \{\ell_{1},\ldots, \ell_{r}\}$, and $H:=C_{k_{1}}^{2}\times \cdots \times C_{k_{s}}^{2}$, there exists a skew-symmetric bicharacter $\beta_{H}$ of $H$ such that $\beta_{H}=\beta_{G}\mid_{H}$, and if $m:=k_{1}\cdots k_{s}$, then $M_{m}(K)$ admits a structure of $H$-regular grading with skew-symmetric bicharacter $\beta_{H}$.
		\end{enumerate}
	\end{Lemma}
	
	\begin{proof}
		(a): For each $1\leq t\leq r$, let $\varepsilon_{\ell_{t}}$ be a primitive $\ell_{t}$-th root of unity and denote by $a_{t}$, $b_{t}$ the generators of $C_{\ell_{t}}^{2}$. Consider the following skew-symmetric bicharacter of $C_{\ell_{t}}^{2}$:
		\[
		\beta_{\ell_{t}}\colon C_{\ell_{t}}^{2}\times C_{\ell_{t}}^{2} \rightarrow K^{\ast},\quad (a_{t}^{i}b_{t}^{j},a_{t}^{k}b_{t}^{l})\mapsto \varepsilon_{\ell_{t}}^{jk-il}.
		\]
		Then $\beta_{G}:= \beta_{\ell_{1}}\otimes \cdots \otimes\beta_{\ell_{r}}$ is a skew-symmetric bicharacter of $G$. By construction, if $g\in C_{\ell_{t}}^{2}$ and $h\in C_{\ell_{s}}^{2}$, with $t\neq s$, then $\beta_{G}(g,h)=1$. Moreover, since $\cM_{\beta_{G}}=\cM_{\beta_{\ell_{1}}}\otimes \cdots \otimes \cM_{\beta_{\ell_{r}}}$ and, by \cite[Proposition 2.3]{bcommutation}, $\det \cM_{\beta_{\ell_{i}}}\neq 0$ for all $1\leq i\leq r$, it follows that $\det \cM_{\beta_{G}}\neq 0$. Therefore, $\beta_{G}$ is non-degenerate.
		
		(b): Let $R:=M_{n}(K)$. We observe that, since $\ell_{1}\cdots \ell_{r}=n$, we have
		\[
		R\cong M_{\ell_{1}}(K)\otimes \cdots \otimes M_{\ell_{r}}(K).
		\]
		
		For each $1\leq t\leq r$, let $P_{t}$ and $Q_{t}$ be the clock and shift matrices, respectively. Define $X_{t,1}:=I_{\ell_{1}}\otimes \cdots \otimes I_{\ell_{t-1}}\otimes P_{t}\otimes I_{\ell_{t+1}}\cdots \otimes I_{\ell_{r}}$ and $X_{t,2}:=I_{\ell_{1}}\otimes \cdots \otimes I_{\ell_{t-1}}\otimes Q_{t}\otimes I_{\ell_{t+1}}\cdots \otimes I_{\ell_{r}}$. Then, given $g\in G$, with $g=a_{1}^{s_{1,1}}b_{1}^{s_{1,2}}\cdots a_{t}^{s_{t,1}}b_{t}^{s_{t,2}}$, we define
		\[
		R_{g}=\operatorname{Span}_{K}\{X_{1,1}^{s_{1,1}}X_{1,2}^{s_{1,2}}\cdots X_{t,1}^{s_{t,1}}X_{t,2}^{s_{t,2}}\}.
		\]
		By \cite[Theorem 2.15]{Kochetov.Elduque.Book} it follows that $R=\bigoplus_{g\in G} R_{g}$ is a $G$-regular grading for $R$ with skew-symmetric bicharacter $\beta_{G}$.
		
		(c): Finally, given $1\leq s\leq r$, $k_{1},\ldots, k_{s}\in \{\ell_{1},\ldots, \ell_{r}\}$, and $H:=C_{k_{1}}^{2}\times \cdots \times C_{k_{s}}^{2}$, we define
		\[
		\beta_{H}:= \beta_{k_{1}}\otimes \beta_{k_{2}}\otimes \cdots \otimes \beta_{k_{s}}.
		\]
		Then $\beta_{H}$ is a skew-symmetric bicharacter of $H$, $\beta_{H}=\beta_{G}\mid_{H}$, and, using the same arguments as those used for the group $G$, one can see that $M_{m}(K)$ admits a structure of $H$-regular grading with skew-symmetric bicharacter $\beta_{H}$.
	\end{proof}
	
The next example illustrates an application of Lemma \ref{realization} and provides an alternative to the method of \cite[Corollary 62]{B.L.K} for constructing a regular grading on a semisimple algebra whose simple components are all matrix algebras of prime power order.
	
	\begin{Example}
		Consider $A_{1}=M_{n_{1}}(K)$ with $n_{1}=p^{\ell_{1}}$, $A_{2}=M_{n_{2}}(K)$ with $n_{2}=p^{\ell_{2}}$, and $A_{3}=M_{n_{3}}(K)$, with $n_{3}=p^{\ell_{3}}$,  where $p$ is a prime number. Suppose $\ell_{1}\geq \ell_{2}\geq \ell_{3}$. In \cite[Corollary 62]{B.L.K} it was shown that $A:=A_{1}\oplus A_{2}\oplus A_{3}$ admits a quantum regular decomposition by using the Kronecker product. Let us see another way to do it in this example.  Write $n_{1}=p^{\ell_{1}-\ell_{2}}p^{\ell_{2}-\ell_{3}}p^{\ell_{3}}$ and consider 
		\[
		G:=(C_{p^{\ell_{1}-\ell_{2}}})^{2}\times (C_{p^{\ell_{2}-\ell_{3}}})^{2}\times (C_{p^{\ell_{3}}})^{2}.
		\]
		
		By Lemma \ref{realization}, consider the $G$-regular grading on $A_{1}$ with skew-symmetric bicharacter $\beta_{G}$ given by 
		\[
		\beta_{G}(a_{1}^{i_{1}}b_{1}^{j_{1}}a_{2}^{i_{2}}b_{2}^{j_{2}}a_{3}^{i_{3}}b_{3}^{j_{3}}, a_{1}^{k_{1}}b_{1}^{l_{1}}a_{2}^{k_{2}}b_{2}^{l_{2}}a_{3}^{k_{3}}b_{3}^{l_{3}})=\varepsilon_{1}^{j_{1}k_{1}-i_{1}l_{1}}\varepsilon_{2}^{j_{2}k_{2}-i_{2}l_{2}}\varepsilon_{3}^{j_{3}k_{3}-i_{3}l_{3}}
		\]
		where $\varepsilon_{1}$ is a primitive $p^{\ell_{1}-\ell_{2}}$-th root of unity, $\varepsilon_{2}$ is a primitive $p^{\ell_{2}-\ell_{3}}$-th root of unity and $\varepsilon_{3}$ is a primitive $p^{\ell_{3}}$-th root of unity. Now, we set $H:=(C_{p^{\ell_{2}-\ell_{3}}})^{2}\times C_{p^{\ell_{3}}}^{2}$ and $L:=C_{p^{\ell_{3}}}^{2}$, then, by Lemma \ref{realization} we can consider the structure of $H$ and $L$-regular gradings on $A_{2}$ and $A_{3}$ with skew-symmetric bicharacters $\beta_{H}$ and $\beta_{L}$, respectively, such that $\beta_{H}=\beta_{G}\mid_ {H}$ and $\beta_{L}=\beta_{G}\mid_{L}=(\beta_{H})\mid_{L}$. In this way, $A_{1}$, $A_{2}$ and $A_{3}$ become $G$-graded $\beta_{G}$-commutative algebras by setting 
		\[
		(A_{i})_{g}=\begin{cases}
			(A_{i})_{h_{i}}\quad \text{if}\quad g=h_{i}\in H_{i}\\
			0,\quad \text{otherwise}
		\end{cases},\quad\quad 1\leq i\leq 3
		\]
		and for any $g\in G$ if we put $A_{g}=(A_{1})_{g}\oplus (A_{2})_{g}\oplus (A_{3})_{g}$, it follows that $A=\oplus_{g\in G}A_{g}$ is a $G$-graded regular algebra with skew-symmetric bicharacter $\beta_{G}$. 
	\end{Example}

    \begin{Lemma}\label{lYB}
\begin{enumerate}
  \item Let  $G$ be a group, $R$ be a regularly $G$-graded algebra, $S$ any $G$-graded subalgebra of $R$. Then $R\oplus S$ is a $G$-graded regular algebra.
  \item Let $R$ be a  regularly $G$-graded algebra with skew-symmetric bicharacter $\beta$. Suppose that $S_{1},\ldots,S_{r}$ are $G$-graded subalgebras of $R$. Then each $S_{i}$ is a $G$-graded $\beta$-commutative algebra and $R\oplus S_{1}\oplus\cdots\oplus S_{r}$ is a $G$-graded regular algebra with skew-symmetric bicharacter $\beta$.
\end{enumerate}
  \end{Lemma}
\begin{proof}
  Both claims are immediate from the definitions.
\end{proof}

Now if $n\in\mathbb{N}$ and $n=p_1\cdots p_s$, where each $p_i$ is a prime number, $G=C_{p_1}^2\times\cdots\times C_{p_s}^2$, then by Lemma \ref{realization}, if we write $M_n(K)=M_{p_1}(K)\otimes\cdots\otimes M_{p_s}(K)$, it follows that $M_{n}(K)$ has a structure of $G$-graded regular algebra with a skew-symmetric bicharacter given by $\beta_{G}:=\beta_{p_{1}}\otimes\cdots\otimes\beta_{p_{s}}$. Now if $m\mid n$ and $m=p_{1}\cdots p_{t}$ where $t\le s$ then $M_{m}(K)=M_{p_1}(K)\otimes\cdots\otimes M_{p_t}(K)$ is isomorphic to a $G$-graded subalgebra of $M_{n}(K)$ via the graded embedding 
\[
x_1\otimes\cdots\otimes x_t\to x_1\otimes\cdots\otimes x_t\otimes \underbrace{1\otimes\cdots\otimes 1}_{s-t}\mbox{where each $x_i$ is homogeneous}.
\]
Using Lemma \ref{lYB}, we get 
\begin{Proposition} 
		\label{realization.1}
		Let $R=M_{n_{1}}(K)\oplus \cdots \oplus M_{n_{r}}(K)$ be a finite-dimensional semisimple algebra such that $n_{i}\mid n_{1}$, for any $1\leq i\leq r$. Then, $R$ admits the structure of a $G$-graded regular algebra for some finite abelian group $G$. 
\end{Proposition}
	
Now we are in a position to prove the main result of the paper.

	\begin{proof}[Proof of Theorem \ref{main.theorem}] Suppose that $R=M_{n_{1}}(K)\oplus \cdots \oplus M_{n_{r}}(K)$ admits a regular quantum commutative decomposition $\mathscr{R}\colon R=R_{1}\oplus \cdots \oplus R_{m}$ with multiplication function $\theta(\cdot,\cdot)$. If there exist $1\leq i\neq j\leq m$ such that $\theta(i,\ell)=\theta(j,\ell)$ for all $1\leq \ell\leq m$, then we may replace the two summands $R_{i}$ and $R_{j}$ by $S:=R_{i}\oplus R_{j}$. Hence, we may assume without loss of generality that $\mathscr{R}$ is minimal. Denote by $e_{1}$,\dots, $e_{r}$ the primitive central idempotents associated to the semisimple decomposition of $R$. Of course, we can take $n_{1}\geq n_{i}$, for all $1\leq i\leq r$. 
		By \cite[Proposition 46]{B.L.K} and \cite[Theorem 52]{B.L.K}, for any $1\leq i\leq r$, there exists $i_{1}$,\dots $i_{k}$ such that 
		\[
		M_{n_{i}}(K)=R_{i_{1}}e_{i}\oplus \cdots \oplus R_{i_{k}}e_{i}
		\]
		is a regular grading for some finite abelian group $H_{i}$, and  $|H_{i}|=k=n_{i}^{2}$. Since $\mathscr{R}$ is minimal of quantum length $m$ and $n_{1}\geq n_{i}$ for all $1\leq i\leq r$, Proposition \ref{necessary.condition} yields $|H_1|=m$. Furthermore, for any $1\leq i\leq r$, if $\beta_i$ denotes the skew-symmetric bicharacter on $H_i$ corresponding to the $H_{i}$-regular grading  of $M_{n_{i}}(K)$, then $\cM_{\beta_i}$ is a submatrix of $\cM_{\beta_1}$ and by Corollary \ref{Bicharpres} we conclude $n_{i}\mid n_{1}$. The converse is true because of Proposition \ref{realization.1}.
	\end{proof}
	\section{Applications}
	
	
	It was shown by Bahturin and Parmenter in \cite[Corollary 4.1]{Bahturin.Parmenter} that, for any finite group $G$, the group algebra $KG$ admits a $Q$-grading by a finite abelian group $Q$ such that it becomes a $\beta$-commutative algebra for an appropriate skew-symmetric bicharacter $\beta$ of $Q$. Nevertheless, this is no longer true in the context of regular gradings. For instance, as shown in \cite[Example 59]{B.L.K}, the group algebra $KS_{5}$ does not admit the structure of a regular quantum commutative algebra.
	
	Given a finite group $G$ we define the following set:
	\[
	\mathbf{d}(G):=\{\deg(\chi)\mid \text{$\chi$ is an irreducible character of $G$}\}.
	\]
	We say that $G$ is a \emph{peak group} if $\mathbf{d}(G)$ has the greatest element $d$ with respect to the divisibility ordering. A character $\chi$ with $\deg(\chi)=d$ is called a \emph{peak character}. The next result, which follows directly from Theorem~\ref{main.theorem}, provides a necessary and sufficient condition under which a group algebra admits a regular grading by a finite abelian group.

	\begin{Proposition}
		\label{criterion.group.algebra}
		Let $G$ be a finite group. The group algebra $KG$ admits a regular quantum commutative decomposition if and only if $G$ is a peak group. \hfill$\Box$
	\end{Proposition}
	Given a finite group $G$ with a neutral element $e\in G$, let us recall some basic definitions.
	\begin{enumerate}
		\item[(1)]  We define the \emph{upper central series} of $G$ as the sequence of normal subgroups of $G$:
		\[
		\{ e\}=Z_{0}\subseteq Z_{1}\subseteq\cdots \subseteq Z_{n}\subseteq\cdots  
		\]
		where  $Z_{1}:=Z(G)$, and for $j\geq 1$, $Z_{j}$ is defined inductively by the condition $Z_{j+1}/Z_{j}=Z(G/Z_{j})$. We say that $G$ is \emph{nilpotent} if there exists $n\in \mathbb{N}$ such that $Z_{n}=G$.
		\item[(2)] A finite group $G$ is called solvable if it has a finite chain of subgroups 
		\[
		\{e\}=G_{0}\subseteq G_{1}\subseteq \cdots \subseteq G_{n}=G
		\]
		such that $G_{i}$ is a normal subgroup of $G_{i+1}$ for any $0\leq i\leq n-1$ and each factor group $G_{i+1}/G_{i}$ is abelian, for any $0\leq i\leq n-1$.
		\item[(3)] A solvable group $G$  is called supersolvable if it has a series, as above, such that each $G_{i+1}/G_{i}$ is cyclic for any $1\leq i\leq n-1$.
       \item[(4)] {Given characters $\chi$ and $\varphi$ of $G$, the Hermitian inner product of $\chi$ and $\varphi$ is defined by
    \[
    \langle \chi\mid \varphi\rangle=\dfrac{1}{|G|}\sum_{x\in G}\chi(x)\varphi(x^{-1}).
    \]

    \item[(5)] Given a normal subgroup $N\subseteq G$ and a character $\theta$ of $N$, and $g\in G$, we define the \emph{conjugate character} $\theta^{g}$ by
    \[
    \theta^{g}(x)=\theta(gxg^{-1}),\quad \text{for all}\quad x\in N.
    \]

    \item[(6)] Given a normal subgroup $N\subseteq G$ and an irreducible character $\theta$ of $N$, the subgroup
    \[
    I_{G}(\theta)=\{g\in G\mid \theta^{g}=\theta\}
    \]
    is called the \emph{inertia group} of $\theta$ in $G$. Clearly, $N\subseteq I_{G}(\theta)$.

    \item[(7)] Let $H\subseteq G$ be a subgroup of $G$. Given a nontrivial irreducible character $\lambda$ of $H$, we define the induced character $\lambda^{G}$ of $G$ by
    \[
    \lambda^{G}(g)=\dfrac{1}{|H|}\sum_{x\in G}\lambda^{o}(xgx^{-1}),
    \]
    where
    \[
    \lambda^{o}(g)=
    \begin{cases}
    \lambda(g), & g\in H,\\
    0, & g\notin H.
    \end{cases}
    \]
    \item[(8)]  If $H$ is a subgroup of $G$ and $\rho$ is a representation of $G$ with character $\chi$ we denote by $\chi_{H}$ the character of the restriction representation $\rho_{H}$. }
	\end{enumerate}

	\subsection{Nilpotent groups}

	Let $G$ be a finite nilpotent group of order $|G|=p_{1}^{\ell_{1}}\cdots p_{n}^{\ell_{n}}$, where $p_{1}$,\dots, $p_{n}$ are distinct prime numbers and $\ell_{1}$,\dots $\ell_{n}\in \mathbb{N}$ . By \cite[Theorem 6.12]{Curtis.Reiner} 
	\[
	G\cong P_{1}\times \cdots \times P_{n}
	\]
	where $P_{i}$ is the Sylow $p_{i}$-subgroup of $G$, for all $1\leq i\leq n$. 
	Any irreducible character $\chi$  of $G$ decomposes as the tensor product  $\chi\cong \lambda_{1}\otimes \cdots \otimes \lambda_{n}$, where $\lambda_{i}$ is an irreducible character of $P_{i}$, for all $1\leq i\leq n$. In particular
	\[
	\deg(\chi)=\deg(\lambda_{1})\cdots \deg(\lambda_{n}). 
	\]
	Since for any $1\leq i\leq n$, $\deg(\lambda_{i})\mid |P_{i}|$ and $|P_{i}|=p_{i}^{\ell_{i}}$, we conclude that there exists $k_{1}$,\dots, $k_{n}\in \mathbb{Z}_{\geq 0}$ such that
	\[
	\deg(\chi)=p_{1}^{k_{1}}\cdots p_{n}^{k_{n}}. 
	\]
	Now, let $\psi$ be a character of $G$ such that $\deg(\psi)=\max \mathbf{d}(G)$, and write  $\deg(\psi)=p_{1}^{t_{1}}\cdots p_{n}^{t_{n}}$, for some $t_{1}$,\dots, $t_{n}\in \mathbb{Z}_{\geq 0}$. 
	
	We claim that for any irreducible character $\chi$ of $G$ we have $\deg(\chi)\mid \deg(\psi)$. Indeed,  suppose there exists an irreducible character $\widetilde{\psi}$ of $G$ such that $\deg(\widetilde{\psi})\nmid \deg(\psi)$. In this case, if $\deg(\widetilde{\psi})=p_{1}^{\widetilde{t}_{1}}\cdots p_{n}^{\widetilde{t}_{n}}$, then there exists $1\leq j\leq n$, such that $\widetilde{t}_{j}>t_{j}$. Now, we write
	\[
	\psi=\lambda_{1}\otimes \cdots \otimes \lambda_{n}
	\]
	and 
	\[
	\widetilde{\psi}=\widetilde{\lambda}_{1}\otimes \cdots \otimes \widetilde{\lambda}_{n}
	\]
	where for any $1\leq i\leq n$,  $\lambda_{i}$, $\widetilde{\lambda}_{i}$ are irreducible characters of $P_{i}$ with $\deg(\lambda_{i})=p_{i}^{t_{i}}$ and $\deg(\widetilde{\lambda}_{i})=p_{i}^{\widetilde{t}_{i}}$.  Let us define the following character of $G$
	\[
	\phi:=  \lambda_{1}\otimes \cdots \lambda_{j-1}\otimes \widetilde{\lambda}_{j}\otimes \lambda_{j+1}\otimes\cdots \otimes \lambda_{n}.
	\]
	By \cite[Theorem 10]{serre1977linear},  $\phi$ is an irreducible character of $G$ and 
	\begin{align*}
		\deg(\phi) &=p_{1}^{t_{1}}\cdots p_{j-1}^{t_{j-1}}p_{j}^{\widetilde{t}_{j}} p_{j+1}^{t_{j+1}}\cdots p_{n}^{t_{n}}\\
		&> 	p_{1}^{t_{1}}\cdots p_{j-1}^{t_{j-1}}p_{j}^{t_{j}}p_{j+1}^{t_{j+1}}\cdots p_{n}^{t_{n}}=\deg(\psi),
	\end{align*}
	which is a contradiction. Thus, for any character $\chi$  of $G$ we have $\deg(\chi)\mid \deg(\psi)$. 
	
	The above argument can be stated as follows.
	
	\begin{Proposition}
		\label{pNG} Any finite nilpotent group is a peak group.
	\end{Proposition}
	\begin{Remark}
		The converse of the above Proposition is not true. Indeed, $S_3$ is not nilpotent but peak; the group algebra $KS_{3}$ admits a regular quantum commutative decomposition.
	\end{Remark}

	\subsection{Metacyclic groups}
{
    Recall that a finite group $G$ is called \emph{metacyclic} if it contains a normal cyclic subgroup $A$ such that the quotient group $G/A$ is also cyclic.

\begin{Example}
The following are classical examples of metacyclic groups.
\begin{enumerate}
    \item[(a)] The \emph{dihedral group} of order $2n$, $D_{2n}=\langle a,b \mid a^{n}=b^{2}=1,\; bab=a^{-1}\rangle.$
    \item[(b)] The \emph{dicyclic group} of order $4n$, $ DC_{4n}=\langle a,b,c \mid a^{2}=b^{2}=c^{n}=abc\rangle.$
    \item[(c)] Every group of \emph{square-free order}, that is, every finite group whose order is divisible by no prime square.
\end{enumerate}

For further background on metacyclic groups, we refer the reader to
\cite{Isaacs2008,Robinson1982,Rotman}.
\end{Example}
\begin{Theorem}
Any metacyclic group is a peak group.
\end{Theorem}
\begin{proof}
 Since $G$ is a metacyclic group we have 
 \[
 G=\langle a,b\,|\, a^m=1, b^n=a^r, bab^{-1}=a^s\rangle.
 \]
To know the degrees of the characters of irreducible representations of $G$ in which $a$ acts nontrivially, we need to compute the lengths of the orbits of the action of $\langle b\rangle$  on the group $\widehat{N}$ of $1$-dimensional characters of the normal cyclic subgroup $N=\langle a\rangle$ of $G$.  The mapping $a$ to the character $\chi_0$ such that $\chi_0(a)=e^{\frac{2\pi i}{m}}$ extends to an isomorphism of $\langle b\rangle$-groups. So the degrees of the characters of such irreducible representations are just the sizes of the conjugacy classes $\mathrm{cl}(a^k)$ of $G$ inside $N$.  Let $d_{0}$ denote the size of the conjugacy class of $a$. Since conjugation by $b$ acts on $N=\langle a\rangle$ via $a\mapsto a^s$, i.e., $bab^{-1}=a^s$, we have $d_0=\operatorname{ord}_{m}(s)$, where $\operatorname{ord}_m(s)$ denotes the multiplicative order of $s$ modulo $m$. In general, for $a^k\in N$, let $d_k$ denote the size of its conjugacy class. Then $d_{k}$ is the least positive integer $d$ such that $a^{ks^d}=a^k$. Equivalently, $m\mid k(s^d-1)$, and hence $
d_{k}=\operatorname{ord}_{m/\gcd(k,m)}(s)$. Since $m/\gcd(k,m)$ divides $m$, the congruence $s^{d_0}\equiv 1\pmod m$ implies $s^{d_0}\equiv 1\pmod{m/\gcd(k,m)}$. Therefore, by the definition of the multiplicative order, $
d_k=\operatorname{ord}_{m/\gcd(k,m)}(s)\mid d_0$. Thus, the size of every conjugacy class contained in $N$ divides the size of the conjugacy class of $a$.
\end{proof}	

\begin{Corollary} Dihedral groups, dicyclic groups, and groups of square-free order are peak groups.
\end{Corollary}
 }  
	
	\subsection{Supersolvable groups} 
	
	Since nilpotent groups and metacyclic groups are examples of supersolvable groups, it is natural to ask whether for every supersolvable group $G$, $KG$ admits a regular quantum commutative decomposition. However, this is not the case.
	
	\begin{Example} In the GAP library, consider $G=\texttt{SmallGroup}(54,8)$, which is a group of order $54$. Then, by \cite[Corollary 4.5]{pinnock1998supersolubility}, $G$ is a supersolvable group. However, using GAP (\cite{GAP4}), one can verify that $\mathbf{d}(G)=\{1,2,3\}$, and therefore by Corollary \ref{criterion.group.algebra},  $KG$ does not admit a regular quantum commutative decomposition.
	\end{Example}
	\subsection{Solvable groups}
	
	As seen above, the group algebra of a supersolvable, hence solvable, group does not necessarily admit a regular quantum commutative decomposition.  Another important example is $KS_{4}$, which does not admit a regular quantum commutative decomposition because $\mathbf{d}(S_{4})=\{1,2,3\}$.   The goal of this section is to investigate the converse of this result.

	\begin{Lemma}
		\label{two.cases}
		Let $G$ be a finite group and let $\psi$ be a peak character of $G$. Suppose that  $\psi$ satisfies one of the two conditions
		\begin{enumerate}
			\item[(a)] $\deg(\psi)=p^{n}$
			\item[(b)] $\deg(\psi)=p^{n}q^{m}$
		\end{enumerate}	
		where $p$ and $q$ are distinct prime numbers and $n$,$m\in \mathbb{N}$. Then $G$ is a solvable group.
	\end{Lemma}
	\begin{proof} First, suppose that $\deg(\psi)=p^{n}$. In this case, by the divisibility hypothesis, $p\mid \deg(\chi)$ for every irreducible non-linear character $\chi$ of $G$. Then, by \cite[Theorem 6.9]{Isaacs1976}, it follows that $G$ has a normal abelian $p$-complement; that is, there exists a normal abelian subgroup $N$ of $G$ such that $p\nmid |N|$ and $[G:N]=p^{\ell}$, for some $\ell\in \mathbb{N}$. Since $G/N$ is a $p$-group and $N$ is abelian, both are solvable groups. Therefore, $G$ is solvable.
		
		Now, suppose that $\deg(\psi)=p^{n}q^{m}$. Since $\psi$ is a peak character, it follows that for any irreducible character $\chi$ of $G$, there exist $r_{1}, r_{2}\in \mathbb{Z}_{\geq 0}$ such that $\deg(\chi)=p^{r_{1}}q^{r_{2}}$. Thus, $\mathbf{d}(G)\subseteq \{p^{r}q^{s}\mid r,s\in \mathbb{Z}_{\geq 0}\}$. Since the degrees of the irreducible characters of $G$ divide $|G|$, it follows that $|G|$ is of the form $|G|=p^{\ell_{1}}q^{\ell_{2}}m$, for some $\ell_{1}, \ell_{2}\in \mathbb{N}$ and $m\in \mathbb{N}$ coprime to $p$ and $q$. 
		
		Let $p_{1}$ be a prime number such that $m=p_{1}^{\ell_{3}}m'$, where $\ell_{3}\in \mathbb{N}$ and $m'\in \mathbb{N}$ coprime to $p_1$. Since $p_{1}$ does not divide any element of $\mathbf{d}(G)$, by Itô--Michler's theorem \cite[Theorem 7.1]{Navarro2018}, $G$ contains a normal abelian Sylow $p_{1}$-subgroup $P_{1}$. Clearly, $\mathbf{d}(G/P_{1})\subseteq \mathbf{d}(G)$. Repeating this process, if necessary, we may assume without loss of generality that $|G|=p^{\ell_{1}}q^{\ell_{2}}p_{1}^{\ell_{3}}$. Hence, $G/P_{1}$ is a group of order $p^{\ell_{1}}q^{\ell_{2}}$, and by Burnside's Theorem \cite[Theorem 7.8]{Isaacs2008}, it follows that $G/P_{1}$ is solvable. Therefore, $G$ is also solvable.
	\end{proof}

	For the next result we recall that a group $G$ has a \emph{normal $p$-complement}, where $p$ is a prime number, if  there exists a normal subgroup $N$ of $G$ such that $p\nmid |N|$ and $[G:N]=p^{k}$, for some $k\in \mathbb{N}$.
	
	\begin{Lemma}
		\label{normal.p.complement}
		
		Let $G$	be a peak group and $\psi$ be a peak character of $G$ such that $\deg(\psi)=p_{1}^{n_{1}}p_{2}^{n_{2}}p_{3}^{n_{3}}$, where $p_{1}$, $p_{2}$ and $p_{3}$ are distinct prime numbers and $n_{1}$, $n_{2}$, $n_{3}\in \mathbb{N}$. Suppose $G$ has a normal $p_{i}$-complement, for some $i\in \{1,2,3\}$. Then, $G$ is solvable. 
	\end{Lemma}

	\begin{proof} Without loss of generality we assume $i=1$ and let $N$ be a normal $p_{1}$-complement of $G$.  By Clifford's theorem \cite[Theorem 6.2]{Isaacs1976} there exists an irreducible representation $\nu_{\psi}$ of $N$ such that the restriction $\psi|_N$ of $\psi$ to $N$ can be written in the form
		\[
		\psi|_N=k\sum_{i=1}^{t}\nu_{i}
		\]
		where $\nu_{\psi}=\nu_{1}$,\dots, $\nu_{t}$ are the distinct conjugates of $\nu_{\psi}$ in $G$,  $k=k(\nu_{\psi})\in \mathbb{N}$, and $t=[G:I_{G}(\nu_{\psi})]$. Hence
		\[
		\deg(\psi)=kt\deg(\nu_{\psi}). 
		\]
		Similarly, given any irreducible character $\chi$ of $G$ there exists an irreducible representation $\nu_{\chi}$ of $N$ and $k'$,$t'\in \mathbb{N}$ such that 
		\[
		\deg(\chi)=k't'\deg(\nu_{\chi}).
		\]
		By \cite[Theorem 5.12]{Navarro2018} it follows that $kt$ and $k't'$ divide $[G:N]$, i.e., they are powers of $p_{1}$. Thus, since $\deg(\chi)\mid \deg(\psi)$ and $p_{1}\nmid |N|$ we conclude $\deg(\nu_{\chi})\mid \deg(\nu_{\psi})$.   Now, by Frobenius reciprocity \cite[Lemma 5.2]{Isaacs1976} every irreducible representation of $N$ is equivalent to a direct summand of the restriction to $N$ of some irreducible representation of $G$.  It now follows that for every irreducible representation $\varphi$ of $N$, we have $\deg(\varphi)\mid \deg(\nu_{\psi})$. Hence, $N$ is a peak group.  On the other hand, since $\deg(\nu_{\psi})\mid |N|$ and $\deg(\nu_{\psi})\mid \deg(\psi)$ we conclude $\deg(\nu_{\psi})=p_{2}^{\ell_{2}}p_{3}^{\ell_{3}}$, for some $\ell_{2}$, $\ell_{3}\in \mathbb{Z}_{\geq 0}$. Thus, by Lemma \ref{two.cases} we conclude $N$ is solvable, and since $G/N$ is a $p_{1}$-group, we conclude that $G$ is solvable.
	\end{proof}
	
	The lemmas \ref{two.cases} and \ref{normal.p.complement}, together with Theorem \ref{main.theorem}, yield the following result.
	
	\begin{Proposition} \label{p3conditions}
		\label{criterion}
		A peak group $G$  with a peak character $\psi$ is solvable provided that one of the following is true.
		\begin{enumerate}
			\item[(a)] $\deg(\psi)=p^{n}q^{m}$, where $p$ and $q$ are distinct prime numbers and $n$, $m\in \mathbb{Z}_{\geq 0}$;
			\item[(b)] $\deg(\psi)=p_{1}^{n_{1}}p_{2}^{n_{2}}p_{3}^{n_{3}}$  where $p_{1}$, $p_{2}$ and $p_{3}$ are distinct prime numbers and $n_{1}$, $n_{2}$, $n_{3}\in \mathbb{N}$, and $G$ has a normal $p_{i}$-complement for some $i\in \{1,2,3\}$. 
			\item[(b')] $\deg(\psi)=p_{1}^{n_{1}}p_{2}^{n_{2}}p_{3}^{n_{3}}$  where $p_{1}$, $p_{2}$ and $p_{3}$ are distinct prime numbers and $n_{1}$, $n_{2}$, $n_{3}\in \mathbb{N}$, and there exists $i\in \{1,2,3\}$ such that $p_{i}\mid \deg(\chi)$ for any nonlinear irreducible character $\chi$ of $G$.  
		\end{enumerate}	
	\end{Proposition}
	
	\begin{proof}
		Only $(\mathrm{b}^\prime)$ needs to be proven. So, let $G$ be a finite group, and let $i\in \{1,2,3\}$ such that $p_{i}\mid \deg(\chi)$ for any nonlinear irreducible character of $G$.  By Thompson's theorem \cite[Corollary 12.2]{Isaacs1976}, $G$ has a normal $p_{i}$-complement $N$. It remains to apply (b).
	\end{proof}
	
	\subsection{Frobenius Groups}

	\begin{Definition} Let $G$ be a finite group with a neutral element $e\in G$ and let $H$ be a subgroup of $G$ with $\{e\}\subsetneq H\subsetneq G$ such that for any $g\in G\setminus H$ 
		\[
		H\cap gHg^{-1}=\{e\}.
		\]
		Then $H$ is called a \emph{Frobenius complement } in $G$ and $G$ is called a \emph{Frobenius group}.
	\end{Definition}
	
	Given a Frobenius group $G$, by \cite[Theorem 7.2]{Isaacs1976} there exists a normal subgroup $N$ of $G$ such that $NH=G$ and $H\cap N=\{e\}$, i.e., $G\cong N\rtimes H$, the semidirect product of $N$ and $H$. In this case, $N$ is called the \emph{Frobenius kernel} of $G$. Consider the following types of  characters of $G$:
	
	\begin{enumerate}
		\item[Type 1.] The irreducible characters of $G$ with $N$ contained in their kernels (these are in one-to-one correspondence with the irreducible characters of $H$, since $G/N\cong H$);
		\item[Type 2.] Given a nontrivial irreducible character $\lambda$ of $N$,  consider the induced character $\lambda^{G}$ of $G$. We notice that $\deg(\lambda^{G})=|H|\deg(\lambda)$.
	\end{enumerate}
	
	By \cite[Theorem 13.8]{Dornhoff1971}, the irreducible characters of Type 1 and Type 2 constitute all the irreducible characters of $G$. Thus, we conclude that
	\begin{equation}
		\label{eq.Frobenius.group'}
		\mathbf{d}(G)=\mathbf{d}(H)\cup \{|H|\deg(\lambda)\mid \lambda\quad \text{is a nontrivial irreducible character of $N$} \}.
	\end{equation}
	In particular, if $N$ is abelian, then
	\begin{equation}
		\label{eq.Frobenius.group}
		\mathbf{d}(G)=\mathbf{d}(H)\cup \{|H|\}.
	\end{equation}

	\begin{Corollary}
		\label{Frobenius.group}
		Any Frobenius group is a peak group. 
	\end{Corollary}
	\begin{proof}Denote by $H$ the Frobenius complement of $G$ and by $N$ the Frobenius kernel of $G$. By \cite[Proposition 16.13(b)]{Huppert1998}, we have $N=\operatorname{Fit}(G)$, where $\operatorname{Fit}(G)$ denotes the Fitting subgroup of $G$, that is, the unique largest normal nilpotent subgroup of $G$. In particular, $N$ is nilpotent, and by Proposition \ref{pNG}, $N$ is a peak group with a peak character $\nu$ of $N$. Consider the induced character $\nu^{G}$ of $G$, which is irreducible since $G$ is a Frobenius group. Let $\chi$ be an arbitrary irreducible character of $G$. By Equation~(\ref{eq.Frobenius.group'}), either $\deg(\chi)=\deg(\varphi)$, where $\varphi$ is an irreducible character of $H$, or $\deg(\chi)=|H|\deg(\lambda)$, where $\lambda$ is an irreducible character of $N$. In the first case, since $\deg(\chi)\mid |H|$, we have
		\[
		\deg(\chi)\mid |H|m
		\]
		for every $m\in \mathbf{d}(N)$. In particular, $\deg(\chi)\mid |H|\deg(\nu)=\deg(\nu^{G})$. 
		
		On the other hand, if $\deg(\chi)=|H|\deg(\lambda)$, then, since $N$ is a peak group, we have
		\[
		\deg(\lambda)\mid \deg(\nu),
		\]
		and therefore
		\[
		\deg(\chi)=|H|\deg(\lambda)\mid |H|\deg(\nu)=\deg(\nu^{G}).
		\]
		By the arbitrariness of $\chi$, we conclude that  $\nu^{G}$ is a peak character of $G$. Hence, $G$ is a peak group. 
	\end{proof}

	Now, consider the special linear group $H:=\operatorname{SL}(2,5)=\{a\in M_{2}(\mathbb{K}_{5})\mid \det(a)=1\}$, and denote its identity element by $e$. It is well known that $H$ has the following presentation:
	\[
	H=\langle x,y,z\mid x^{3}=y^{5}=z^{2}=1,\quad xz=zx,\quad yz=zy,\quad (xy)^{2}=z \rangle.
	\]
	Using the matrices
	\[
	x:=\begin{pmatrix}
		7 & 1\\
		9 & 3
	\end{pmatrix}, \quad y:=\begin{pmatrix}
		3 & 0\\
		0 & 4
	\end{pmatrix},\quad z:=\begin{pmatrix}
		10 & 0\\
		0 & 10
	\end{pmatrix} \in M_{2}(\mathbb{K}_{11}),
	\]
	we see that $\langle x,y,z\rangle$ is a subgroup of $\operatorname{GL}(2,11)$ isomorphic to $H$. In particular, if $V:=\mathbb{K}_{11}^{2}$, then this defines a representation $\rho\colon H\rightarrow \operatorname{GL}(V)$. As shown in \cite[Proposition 18.5]{Passman1968}, if $\rho(h)v=v$ for some $0\neq v\in V$, then $h=e$. Hence, by \cite[Proposition 16.5]{Huppert1998}, the group $G:=V\rtimes H$ is a Frobenius group with Frobenius complement $H$ and Frobenius kernel $V$, where $e\in G$ denotes the identity element. Thus, by Equation (\ref{eq.Frobenius.group}),
	\[
	\mathbf{d}(G)=\mathbf{d}(H)\cup \{|H|\},\quad |H|=|\operatorname{SL}(2,5)|=120.
	\]

	\noindent\underline{\textbf{Claim.}} $G$ has no normal $p$-complements.
	
	\begin{proof}
		Indeed,  suppose that there exists a normal $p$-complement $N$ of $G$. Observe that
		\[
		|G|=|V||H|=2^{3}\cdot 3\cdot 5\cdot 11^{2}.
		\]
		Hence, $p\in \{2,3,5,11\}$. By \cite[Proposition 16.13(a)]{Huppert1998}, either $N$ is a subgroup of $V$, or $V$ is a subgroup of $N$.
		
		If $N$ is a subgroup of $V$, then $|N|\in \{1,11,11^{2}\}$, and in particular $p\neq 11$. However, in either case, $[G:N]$ is not a power of $p$.
		
		Now suppose that $V$ is a subgroup of $N$. In this case, by the Correspondence Theorem, $N/V$ is isomorphic to a normal subgroup of $H$. It is well known that $H$ has exactly three normal subgroups, namely $H$, $Z(H)=\{e,-e\}$, and $\{e\}$. Thus, we have $[N:V]=120$, $[N:V]=2$, or $[N:V]=1$, respectively. Since
		\[
		[G:N]= \dfrac{|G|}{|V|[N:V]},
		\]
		it follows that $[G:N]=1$, $[G:N]=60$, or $[G:N]=120$, respectively. None of these indices is a power of a prime. Therefore, $G$  has no normal $p$-complements.
	\end{proof}

	These calculations show that, although $G$ is a peak group with a peak character $\psi$ satisfying
	\[
	\deg(\psi)=|H|=2^{3}\cdot 3\cdot 5,
	\]
	Proposition \ref{criterion} cannot be applied to $G$, since it has no normal $p$-complements. Moreover, $G$ is not solvable, because $H$ is a non-solvable subgroup of $G$.
	
	We summarize the above discussion in the following proposition.	
	
	\begin{Proposition} Consider the Frobenius group $G=V\rtimes H$, with neutral element $e\in G$, where $V:=\mathbb{K}_{11}^{2}$ and $H:=\operatorname{SL}(2,5)$. Then, we have: 
		\begin{enumerate}
			\item[(a)] $G$ has no normal $p$-complements.
			\item[(b)] $G$ is not a solvable group.
			\item[(c)] $KG$ admits a regular quantum commutative decomposition.
		\end{enumerate}
	\end{Proposition}
	
	The essential reason why $G:=V\rtimes H$ is not solvable and does not satisfy conditions (a) and (b) of the above proposition is the presence of $\operatorname{SL}(2,5)$.

\subsection{Wreath product }
The main goal of this final subsection is to show how to construct peak groups from arbitrary finite groups. In particular, as a consequence, we will provide an example of a group that is neither solvable nor a Frobenius group. To construct such groups, we will use the notion of a wreath product. Let us first recall the definition of a discrete  wreath product. 

Let $H$ and $Q$ be arbitrary groups. Consider the group $P$ of functions with finite support $f:H\to Q$, with usual product of functions. Given $f\in P$ and $h\in H$, we denote by $f^h$ the function given by $f^h(x)=f(xh)$, for any $x\in H$. For each fixed $h\in H$, the map $\hat{h}:f\to f^h$ is an automorphism of $P$ and the map $h\to\hat{h}$ is an injective homomorphism of $H$ to $\mathrm{Aut}(P)$.

 This allows one to form the semidirect product $W=H\ltimes Q$. The elements of $W$ have the form $hf$, where $h\in H$ and $f\in P$. Given $h_1f_1$ and $h_2f_2$ in $W$, one has 
 \[
 (h_1f_1)(h_2f_2)=(h_1h_2)(f_1^{h_2}f_2).
  \]
  One writes $W=H\, \mathrm{wr}\, Q$. The group $P$ is called the \emph{base group} of the wreath product. Every nontrivial normal subgroup of $W$ has nontrivial intersection with $P$. For further details on the wreath products, we refer the reader to \cite{Rotman}.

The following theorem is the main result of this section.

\begin{Theorem}
\label{construct.peak}
Given any finite group $H$ there exists a peak group $G$ such that $G/\operatorname{Fit}(G)\cong H$.

\end{Theorem}
\begin{proof} Let $p$ be a prime number such that $p\nmid |H|$, and consider the cyclic group $C$ of order $p$. We denote by $e_{C}$ and $e_{H}$ the identity elements of $C$ and $H$, respectively. Consider $G:=C\,\mathrm{wr}\, H$. We are going to show that $G/\operatorname{Fit}(G)\cong H$ and that $G$ is a peak group. Since $G/P\cong H$, it is enough to show that $P=\operatorname{Fit}(G)$. Since $P$ is an abelian normal subgroup of $G$, one has $P\subseteq \operatorname{Fit}(G)$.

 To prove the converse, note that $P$ is the Sylow $p$-subgroup of $G$, hence of $\operatorname{Fit}(G)$. We  have $\operatorname{Fit}(G)=P\times S$ where  $S$ is the product of normal Sylow $q$-subgroups of $\operatorname{Fit}(G)\cap H$, with $q\ne p$. Now if $\operatorname{Fit}(G)\ne P$ then there exists $h\in \operatorname{Fit}(G)\cap H$ such that $h\notin P$. In this case, $h$ commutes with every element of $P$. If $h\ne e_H$, take $f\in P$ such that $f(e_H)=e_C$, $f(h)\ne e_C$. Then it follows from $(hf)=fh=hf^h$  that $f=f^h$ and $e_C=f(e_H)=f^h(e_H)=f(h)\ne e_C$, a contradiction. So $P=\operatorname{Fit}(G)$.

Finally, we show that $G$ is a peak group. Let $\theta$ be an irreducible character of $P$ given by
\[
\theta(f)=\chi(f(e_H))
\]
where  $\chi$ is a nontrivial (=faithful) irreducible character of $C$. We will show that $I_G(\theta)=P$.  Since $P\subset I_G(\theta)$, it is enough to show that the action of a nontrivial $h\in H$ on $\theta$ is nontrivial, that is, $\theta\ne\theta^{h}$. But if $f\in P$ is the same element of $P$, as above, then $\theta(f)=\chi(e_C)=1$ while $\theta(f^h)(e_H)=\chi(f(h))\ne 1$ because $\chi$ is a faithful character of $C$. Thus, $I_{G}(\theta)=P$, and by \cite[Theorem 6.11]{Isaacs1976}, we conclude that $\theta^{G}$ is an irreducible character of $G$ of degree $|H|$. Since $[G:\operatorname{Fit}(G)]=|H|$, Itô's theorem \cite[Theorem 6.15]{Isaacs1976} implies that $G$ is a peak group.

\end{proof}

\begin{Example} Since $H$ is arbitrary and $G/\operatorname{Fit}(G)\cong H$, it follows that peak groups need not be solvable. If such nonsolvable group were alwyas Frobenius, then by \cite[Proposition 16.13(b)]{Huppert1998}, its Frobenius complement would be isomorphic to $H$. However, by Zassenhaus's theorem \cite[Theorem 18.6]{Passman1968}, every nonsolvable Frobenius complement contains a subgroup isomorphic to $SL(2,5)$. Since $|SL(2,5)|=120$, taking $H=A_5$ guarantees that $G$ is both nonsolvable and non-Frobenious.
\end{Example}

\end{document}